\documentclass[]{amsart}

\usepackage[dvipdfmx]{graphicx,color}
\usepackage{amsmath,amssymb}

\makeatletter
\@addtoreset{equation}{section}

\makeatother

\newtheorem{theorem}{Theorem}[section]
\newtheorem{proposition}[theorem]{Proposition}
\newtheorem{corollary}[theorem]{Corollary}
\newtheorem{lemma}[theorem]{Lemma}
\newtheorem*{conjecture*}{Conjecture}
\newtheorem*{maintheorem*}{Main Theorem}
\newtheorem*{modifiedconjecture*}{Modified Conjecture}
\theoremstyle{definition}

\theoremstyle{remark}
\newtheorem{remark}[theorem]{Remark}
\newtheorem{example}[theorem]{Example}

\begin{document}

\title[Diameter bound and the Plateau-Douglas problem]{A diameter bound for compact surfaces and the Plateau-Douglas problem}
\author{Tatsuya Miura}
\address{Tokyo Institute of Technology, 2-12-1 Ookayama, Meguro-ku, Tokyo 152-8551, Japan}
\email{miura@math.titech.ac.jp}
\date{\today}

\begin{abstract}
  In this paper we give a geometric argument for bounding the diameter of a connected compact surface (with boundary) of arbitrary codimension in Euclidean space in terms of Topping's diameter bound for closed surfaces (without boundary).
  The obtained estimate is potentially optimal for minimal surfaces in the sense that optimality follows if the Topping conjecture holds true.
  Our result directly implies an explicit nonexistence criterion in the classical Plateau-Douglas problem.
  We exhibit examples of boundary contours to ensure that our criterion is of novel type compared with classical criteria based on the maximum principle and White's criterion based on a density estimate.
\end{abstract}

\maketitle


\section{Introduction}\label{sect:intro}

Let $M^2\hookrightarrow\mathbf{R}^n$ denote a compact two-dimensional surface immersed into Euclidean $n$-space with $n\geq3$, throughout this paper (unless specified).
A well-known conjecture of P.\ Topping \cite{Topping1998} predicts that, at least for $n=3$, among all connected closed surfaces $M$ (without boundary)
\begin{equation}\label{eq:Toppingconjecture}
  C_T(n) :=\inf_M \frac{1}{d(M)}\int_M|H| = \pi,
\end{equation}
where $d(M):=\max_{p,q\in M\hookrightarrow\mathbf{R}^n}|p-q|$ denotes the extrinsic diameter and $H$ denotes the (inward) mean curvature vector with the convention that $|H|\equiv1$ holds for a unit sphere.
An optimal shape is expected to be a long cylinder with capped ends, which directly implies that $C_T\leq\pi$.
On the other hand, as is discussed below, one may naturally expect that among all connected compact minimal surfaces $M'$ (with boundary)
\begin{equation}\label{eq:LengthDiameterconjecture}
  C_{LD}(n) := \inf_{M'} \frac{\ell(\partial M')}{d(M')} = 2,
\end{equation}
where $\ell(\partial M')$ denotes the length of the boundary $\partial M'$ in $\mathbf{R}^n$.
Note that $C_{LD}\leq2$ follows by a flat surface spanning a closed curve consisting of two parallel long segments capped by semicircles (which is optimal in the flat case).
Both the exact values of $C_T$ and $C_{LD}$ are still open due to substantial difficulties such as the nonlocality of diameter.
It is unclear even whether the values depend on $n$.

In this paper we discover that these constants can be directly related in a very simple way, verifying that
\begin{equation}\label{eq:constants}
  \frac{2}{\pi}C_T(n) \leq C_{LD}(n) \leq 2.
\end{equation}
This result not only gives a new explicit lower bound of $C_{LD}$ (since for example $C_T\geq\pi/32$ is already known, see below) but also implies that, somewhat surprisingly, $C_{LD}=2$ would automatically follow if $C_T=\pi$ were verified, thus making those conjectures better grounded.
In particular, our result seems the first attempt to understand the optimal value of $C_{LD}$.
In fact, inequality \eqref{eq:constants} directly follows by our main result that asserts a more general diameter bound for compact surfaces:

\begin{theorem}\label{thm:main}
  Let $M^2\hookrightarrow\mathbf{R}^n$ be an immersed connected compact surface, where $n\geq3$.
  Then
  \begin{equation}\label{eq:main}
    d(M) \leq \frac{1}{C_T(n)}\left(2\int_M|H|+\frac{\pi}{2}\ell(\partial M)\right).
  \end{equation}
\end{theorem}

Now we review relevant previous results more precisely.

The Topping conjecture \cite{Topping1998} originally predicts that $d(M)<\frac{1}{\pi}\int_M|H|$ holds for every connected closed surfaces $M^2\hookrightarrow\mathbf{R}^3$, motivated by L.\ Simon's diameter estimate.
Concerning the lower bound of $C_T$, the weaker but universal estimate $C_T\geq\pi/32$ (independent of $n$) is first proved by Topping himself \cite{Topping2008}, where the intrinsic diameter is bounded in the form of $d_\mathrm{int}\leq c_m\int_{M^m}|H|^{m-1}$ in general (co)dimension.
For low codimensions, this can be now slightly improved by Brendle's recent result \cite{Brendle2021} on the Michael-Simon Sobolev inequality, cf.\ Appendix \ref{sec:appendix}; in particular, $C_T(3)\geq\pi/16$.
On the other hand, the conjectured lower bound $\pi$ is supported by several partial answers, namely for convex surfaces (classically), constant mean curvature surfaces \cite{Topping1997}, and spheres of revolution \cite{Miura2021}.

Concerning the length-diameter ratio for minimal surfaces \eqref{eq:LengthDiameterconjecture}, to the best of the author's knowledge, the nontrivial positivity of $C_{LD}$ is just recently verified by Menne-Scharrer's result \eqref{eq:MenneScharrer} below, while its expected value $C_{LD}=2$ is not explicitly stated in the literature.
However, this expectation seems quite reasonable because it asks an extendibility of codimension from the trivial flat case (in $\mathbf{R}^2$), while such a flat-to-minimal codimension-extension problem is extensively studied for the isoperimetric inequality and now verified up to codimension two by Brendle; see \cite{Brendle2021} and references therein.
In addition, we indicate that estimating $C_{LD}$ is also meaningful in view of nonexistence theory in the classical Plateau-Douglas problem; we discuss this point in Section \ref{sec:Plateau-Douglas} as part of the main contents.

For general surfaces-with-boundary, Menne-Scharrer's recent result \cite{Menne2017} provides a diameter bound in a fairly general framework of varifolds, which is new even for smooth surfaces: For $1<m<n$ and a connected compact $m$-dimensional submanifold $M^m$ of $\mathbf{R}^n$, the intrinsic diameter $d_\mathrm{int}(M)$ has a bound of the form
\begin{align}\label{eq:MenneScharrer}
  d_\mathrm{int}(M)\leq c_m\left( \int_{M}|H|^{m-1} + \int_{\partial M}|H_{\partial M}|^{m-2} \right).
\end{align}
This estimate is consistent with \eqref{eq:main} since if $m=2$, the last term is interpreted as the length $\ell(\partial M)$.
Their result extends both Topping's result for closed submanifolds \cite{Topping2008} and also Paeng's result for geodesically convex compact surfaces ($m=2$) \cite{Paeng2014}.
Their method is based on an oscillatory characterization of the diameter, which allows them to deal with singular objects, albeit $c_m$ is chosen in a somewhat involved way.
Compared to Menne-Scharrer's result, our present study is less general but our focus is optimality rather than generality.
Our main contribution is giving not only explicit constants but also revealing a direct quantitative link between the closed and non-closed cases by a completely different approach.

In fact, our proof is based on the very simple geometric idea to construct a thin closed surface by enclosing a given compact surface and to compare their energies.
The potentially optimal boundary term in \eqref{eq:main} then appears as a singular limit; this part crucially relies on two-dimensionality of surfaces.
The factor $2$ in front of $\int_M|H|$ would be non-optimal but comes from our ``doubling'' procedure.

In Section \ref{sec:enclosing} we prove Theorem \ref{thm:main}.
In Section \ref{sec:Plateau-Douglas} we indicate that Theorem \ref{thm:main} (or Menne-Scharrer's preceding result) directly implies a novel nonexistence criterion in the Plateau-Douglas problem by exhibiting several concrete examples.
In Appendix \ref{sec:appendix}, some improved bounds for $C_T$ are verified for the reader's convenience.

\subsection*{Acknowledgments}

The author would like to thank Peter Topping for his interest on this paper, and also for pointing out that Brendle's recent result implies an improved diameter bound.
He also thanks Genki Hosono and Yuichi Ike for stimulating discussions.
This work is supported by JSPS KAKENHI Grant Numbers 18H03670, 20K14341, and 21H00990, and by Grant for Basic Science Research Projects from The Sumitomo Foundation.

\section{Bounding diameter for surfaces with boundary}\label{sec:enclosing}

In this section we prove Theorem \ref{thm:main}.
The key step is to ensure the following

\begin{proposition}[Construction of closed surfaces]\label{lem:enclosing}
  For any connected compact surface $M$ immersed into $\mathbf{R}^n$ with $\partial M\neq\emptyset$, there exists a sequence $\{\Sigma_k\}_k$ of connected closed surfaces in $\mathbf{R}^n$ such that
  \begin{align}
    \lim_{k\to\infty}\int_{\Sigma_k}|H| &= 2\int_M|H| + \frac{\pi}{2}\ell(\partial M), \label{eq:enclosing1}\\
    \lim_{k\to\infty}d(\Sigma_k) &=d(M). \label{eq:enclosing2}
  \end{align}
\end{proposition}


We first observe that Theorem \ref{thm:main} immediately follows from Proposition \ref{lem:enclosing}.

\begin{proof}[Proof of Theorem \ref{thm:main}]
  Let $M$ be given.
  We may suppose that $\partial M\neq\emptyset$.
  Take a sequence $\{\Sigma_k\}$ in Proposition \ref{lem:enclosing}.
  Then by definition of $C_T(n)$ we have
  $$\frac{1}{d(\Sigma_k)}\int_{\Sigma_k}|H| \geq C_T(n).$$
  Taking the limit that $k\to\infty$, we deduce \eqref{eq:main} from \eqref{eq:enclosing1} and \eqref{eq:enclosing2}.
\end{proof}

In the remaining part we prove Proposition \ref{lem:enclosing}.

The primal idea is to enclose a compact surface $M$ by a thin closed surface $\Sigma_k$.
In the simplest case of codimension one ($n=3$) and $M$ embedded, we may just take $\Sigma_k$ as the boundary of the $1\over k$-neighborhood of $M$ in $\mathbf{R}^3$ for $k\gg1$ as in Figure \ref{fig:enclosing1}.
Then, roughly speaking, the enclosing surface consists of two parts; a ``nearly double-cover'' of $M$ and ``curved half cylinders of radius $1/k$'' along the boundary $\partial M$.
The curvature energy of the former part clearly converges to twice that of $M$ as $k\to\infty$, while the latter yields $\frac{\pi}{2}\ell(\partial M)$ as a singular limit.

\begin{figure}[htbp]
  \includegraphics[width=50mm]{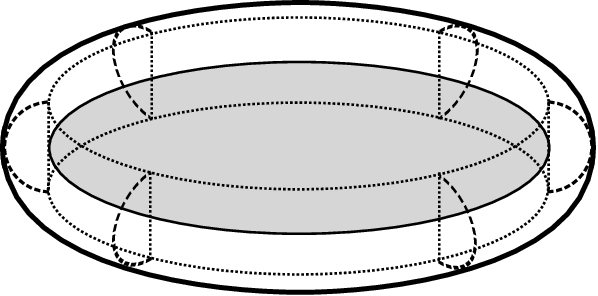}
  \caption{The neighborhood closed surface enclosing a disk.}
  \label{fig:enclosing1}
\end{figure}

The above idea can be extended to non-embedded surfaces almost straightforwardly, but not very directly to higher codimensions since we have no canonical choice of normal directions.
In order to give a unified proof for an arbitrary codimension, we slightly modify our idea; we double $M$ and glue their boundaries via a new surface made by a teardrop-shaped curve (with a cusp) as in Figure \ref{fig:enclosing2}, so that the resulting closed surface looks ``pressed'' to touch $M$ from both sides.
This modification allows us to only think of the boundary construction, which is essentially reduced to the case of codimension one by using a three-dimensional frame associated with each boundary curve.

\begin{figure}[htbp]
  \includegraphics[width=60mm]{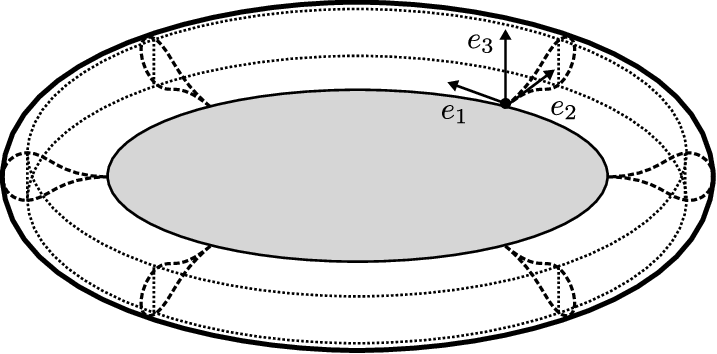}
  \caption{A closed surface enclosing a disk created by a teardrop curve.}
  \label{fig:enclosing2}
\end{figure}

Before entering the proof we prepare a teardrop curve for later use.

\begin{lemma}[Teardrop curve]\label{lem:teardrop}
  There is a sequence $\{\gamma_k\}$ of unit-speed smooth curves $\gamma_k=(x_k,y_k):[0,L_k]\to\mathbf{R}^2$ such that
  \begin{equation}\label{eq:teardrop1}
    \gamma_k(0) =\gamma_k(L_k)=(0,0), \quad \partial_s\gamma_k(0)=-\partial_s\gamma_k(L_k)=(1,0),
  \end{equation}
  \begin{equation}\label{eq:teardrop2}
    \max_{s\in[0,L_k]}|\gamma_k(s)| \leq 2,
  \end{equation}
  \begin{equation}\label{eq:teardrop3}
    \lim_{k\to\infty}\int_{\gamma_k}|\kappa|ds = \pi,
  \end{equation}
  where $\kappa$ and $s$ denote the curvature and the arclength parameter, respectively.
\end{lemma}

\begin{proof}
  By approximation we only need to construct $\gamma_k$ of class $C^{1,1}$ ($=W^{2,\infty}$) and piecewise smooth.
  Define $\gamma_k$ by connecting the origin and the half-circle $C_k:=\{(x-1)^2+y^2=1/k^2,\ x\geq1\}$ via the graph curves $\{y=\pm\frac{1}{k}f(x),\ 0\leq x \leq 1\}$, where $f:[0,1]\to[0,1]$ is any fixed smooth function such that $f\equiv 0$ around $0$ and $f\equiv 1$ around $1$.
  By construction all the assertions are trivial except for \eqref{eq:teardrop3}.
  Convergence \eqref{eq:teardrop3} is also simply confirmed since the total absolute curvature is nothing but the total variation of the tangential angle function; for the half-circle $\int_{C_k}|\kappa|ds=\pi$, while $\frac{1}{k}f\to0$ smoothly so that $\int_{\text{graph}(\pm\frac{1}{k}f)}|\kappa|ds\to0$.
\end{proof}

We are now in a position to give a precise proof of Proposition \ref{lem:enclosing}.

\begin{proof}[Proof of Proposition \ref{lem:enclosing}]
  Fix any connected compact surface $M$ immersed into $\mathbf{R}^n$.
  By approximation it suffices to construct a sequence of surfaces of class $C^{1,1}$ ($=W^{2,\infty}$) and piecewise smooth.

  {\em Step1: Construction.}
  We first prepare a three-dimensional frame $\{e_1,e_2,e_3\}$ along the boundary $\partial M$ for a gluing procedure.
  For each of the boundary components $\partial M_i$, $i=1,\dots,N$, being a closed curve of length $L_i:=\ell(\partial M_i)>0$, we let $c_i:\mathbf{R}/L_i\mathbf{Z}\to\partial M_i\subset\mathbf{R}^n$ be a unit-speed parameterization.
  Now with each $\partial M_i$ we assign the (global) vector fields $e_1:=\partial_sc_i$, the unit tangent, and $e_2$ to be the outward pointing unit conormal of the surface $M$ at the boundary.
  Then we choose one more vector field $e_3$ orthogonal to both $e_1$ and $e_2$.
  Note that such an $e_3$ always exists.
  Indeed, for $n=3$ we may just take $e_3=e_1\times e_2$.
  For $n\geq4$, if we let $N'$ denote the subbundle of the normal bundle to $\partial M_i$ in $\mathbf{R}^n$ defined by $N':=\coprod\operatorname{span}\{e_2\}^\perp$, where $\perp$ stands for the orthogonal complement, then $N'$ is locally trivialized as $I\times\mathbf{R}^{n-2}$, where $I$ is an interval, and hence there is no (topological) obstruction to choose a global section of $N'$ as $n-2\geq2$.
  Slightly abusing the notation, we also mean by $e_j$ just a smooth map from $\mathbf{R}/L_i\mathbf{Z}$ to $\mathbf{S}^{n-1}\subset\mathbf{R}^n$.

  We then define a surface $S$ along the boundary $\partial M$, which will be used to close two copies of $M$.
  Given $\varepsilon>0$, integer $k>0$, and $i=1,\dots,N$, by using the teardrop curve $(x_k,y_k):[0,L_k]\to\mathbf{R}^2$ in Lemma \ref{lem:teardrop}, we define the surface map $S=S_{\varepsilon,k,i}:\mathbf{R}/L_i\mathbf{Z}\times[0,L_k]\to\mathbf{R}^n$ by
  \begin{equation}
    S_{\varepsilon,k,i}(\sigma,s) := c_i(\sigma) + \varepsilon x_k(s)e_2(\sigma) + \varepsilon y_k(s)e_3(\sigma).
  \end{equation}
  Note that since $\{e_1,e_2,e_3\}$ are smooth maps on a compact set, there is some $\bar{\varepsilon}>0$ such that for any $\varepsilon\in(0,\bar{\varepsilon})$ the patch $S$ is regular, i.e., $\partial_s S$ and $\partial_\sigma S$ are linearly independent; the smallness depends only on the frame but not on $k$ thanks to \eqref{eq:teardrop2}.

  Now we define a piecewise smooth closed surface $\Sigma_{\varepsilon,k}$ by gluing the boundaries of two copies of $M$, say $M$ and $M'$, via $S$.
  More precisely, for all $i$ we glue $\partial M_i$ to one boundary component $S(\mathbf{R}/L_i\mathbf{Z}\times\{0\})$, and $\partial M'_i$ to the other $S(\mathbf{R}/L_i\mathbf{Z}\times\{L_\varepsilon\})$, keeping their shapes.
  Notice that $\Sigma_{\varepsilon,k}$ is of class $C^{1,1}$ thanks to boundary condition \eqref{eq:teardrop1}.
  In addition, $\Sigma_{\varepsilon,k}=M\cup M'\cup S$ is connected since $S$ is joined with both $M$ and $M'$ so that $M$ and $M'$ are also joined via $S$.

  {\em Step 2: Quantitative behavior.}
  In what follows we will define $\widetilde{\Sigma}_k:=\Sigma_{\varepsilon_k,k}$ for a well-chosen $\varepsilon_k$ to satisfy the desired properties, \eqref{eq:enclosing1} and \eqref{eq:enclosing2}.

  Concerning diameter, by our construction, clearly $d(M)\leq d(\Sigma_{\varepsilon,k})$ holds.
  In addition, thanks to \eqref{eq:teardrop2}, for any $p\in\Sigma_{\varepsilon,k}$ there is $p'\in M$ such that $|p-p'|\leq 2\varepsilon$, and hence $\max_{p,q\in\Sigma_{\varepsilon,k}}|p-q|\leq \max_{p',q'\in M}|p'-q'| + 4\varepsilon$.
  We thus find that
  \begin{equation}\label{eq:enclosing4}
    |d(\Sigma_{\varepsilon,k}) - d(M)| \leq 4\varepsilon.
  \end{equation}

  Concerning the curvature energy, by definition we have
  \begin{equation}\label{eq:enclosing3}
    \int_{\Sigma_{\varepsilon,k}}|H| = \int_{M}|H| + \int_{M'}|H| + \sum_{i=1}^N \int_{S_{\varepsilon,k,i}}|H| = 2\int_{M}|H| + \sum_{i=1}^N \int_{S_{\varepsilon,k,i}}|H|.
  \end{equation}
  We now confirm that for each $i$ and $k$,
  \begin{equation}\label{eq:enclosing7}
    \lim_{\varepsilon\to0}\int_{S_{\varepsilon,k,i}}|H| = \frac{1}{2}\ell(\partial M_i)\int_{\gamma_k}|\kappa|ds.
  \end{equation}
  To this end we explicitly represent the integral.
  Let $g=(g_{jl})_{1\leq j,l \leq 2}$ denote the pull-back metric of the Euclidean metric of $\mathbf{R}^n$ under $S$ in the local coordinate $(s,\sigma)$; namely, $g_{jl}=\partial_jS\cdot\partial_lS$ where $\partial_1:=\partial_s$ and $\partial_2:=\partial_\sigma$.
  Recall that for $S$ the induced area measure is given by $\sqrt{\det{(g_{jl})}}dsd\sigma$, while the second fundamental form $A$ by $A_{jl}:=(\partial_{jl}S)^\perp$, where $\perp$ denotes the normal projection, and the mean curvature vector $H$ by $H=\frac{1}{2}g^{jl}A_{jl}$ (up to the sign, through the Einstein notation), where $(g^{jl}):=(g_{jl})^{-1}$.
  Thus we have
  \begin{equation}\label{eq:enclosing9}
    \int_{S}|H| = \frac{1}{2}\int_0^{L_i}\int_0^{L_\varepsilon} |g^{jl}A_{jl}| \sqrt{\det{(g_{jl})}}dsd\sigma.
  \end{equation}
  Now we expand the integrand in terms of $\varepsilon$ by explicit calculations.
  Since
  \begin{align*}
    \partial_1S = \varepsilon(\dot x_ke_2+\dot y_ke_3), \quad \partial_2S = e_1+\varepsilon(x_k\dot e_2+y_k\dot e_3),
  \end{align*}
  and since $e_1,e_2,e_3$ are orthogonal, we have as $\varepsilon\to0$,
  $$g_{11}=\varepsilon^2|\dot x_k^2+\dot y_k^2|=\varepsilon^2, \quad g_{22}=1+O(\varepsilon), \quad g_{12}=g_{21}=O(\varepsilon^2);$$
  here and in the sequel asymptotic notations are independent of the parameters $(s,\sigma)$ (but may depend on $k$).
  These imply that
  \begin{equation*}
    \sqrt{\det{(g_{jl})}} = \varepsilon + o(\varepsilon).
  \end{equation*}
  Then, noting that
  $$|g^{jl}A_{jl}| \sqrt{\det{(g_{jl})}} = \frac{1}{\sqrt{\det{(g_{jl})}}}(g_{22}A_{11}+g_{11}A_{22}-g_{12}A_{12}-g_{21}A_{21}),$$
  and that $|g_{11}A_{22}-g_{12}A_{12}-g_{21}A_{21}|=O(\varepsilon^2)$, we now only need to compute the term $g_{22}A_{11}$ ($=O(\varepsilon)$).
  Since $A_{11}=\varepsilon(\ddot x_ke_2^\perp+\ddot y_ke_3^\perp)$, we compute $e_j^{\perp}$ for $j=2,3$.
  In view of the Gram-Schmidt process, if we define the standard projection operator $\pi_u(v):=|u|^{-2}(u\cdot v)u$ (with $\pi_0(v):=0$) and let
  $$u_1:=\partial_1 S, \quad u_2:=\partial_2S-\pi_{u_1}(\partial_2S),$$ 
  then we have the representation
  \begin{align*}
    e_j^\perp &= e_j - \pi_{u_1}(e_j)-\pi_{u_2}(e_j).
  \end{align*}
  By using the above formulae for $\partial_1S$, $\partial_2S$, $g_{11}$, $g_{12}$, $g_{22}$, we first find that $|\pi_{u_2}(e_j)|=O(\varepsilon)$, and then by computing $e_j - \pi_{u_1}(e_j)$ we deduce that
  \begin{align*}
    \left| e_2^\perp - \left( (1-\dot x_k^2)e_2 - \dot x_k\dot y_ke_3 \right) \right| = O(\varepsilon), \quad \left| e_3^\perp - \left( (1-\dot y_k^2)e_3 - \dot x_k\dot y_ke_2 \right) \right| = O(\varepsilon).
  \end{align*}
  Then a direct computation shows that
  \begin{align*}
    \left|\ddot x_ke_2^\perp+\ddot y_ke_3^\perp\right|^2 &= \left|\left( \ddot x_k(1-\dot x_k^2)-\ddot y_k \dot x_k\dot y_k \right) e_2 + \left( \ddot y_k(1-\dot y_k^2)-\ddot x_k\dot x_k\dot y_k \right) e_3 \right|^2+o(1)\\
    &=|\ddot x_k\dot y_k - \ddot y_k\dot x_k|^2+o(1)=|\kappa|^2+o(1),
  \end{align*}
  and hence $|g_{22}A_{11}|=\varepsilon|\kappa|+o(\varepsilon)$.
  In summary, we have
  \begin{equation*}
    |g^{jl}A_{jl}|\sqrt{\det{(g_{jl})}}=|\kappa|+o(1).
  \end{equation*}
  Inserting this into \eqref{eq:enclosing9} and integrating over the domain $[0,L_i]=[0,\ell(\partial M_i)]$ of $\sigma$, we obtain \eqref{eq:enclosing7} in form of
  $$\int_{S_{\varepsilon,k,i}}|H| = \frac{1}{2}\ell(\partial M_i)\int_{\gamma_k}|\kappa|ds + o(1).$$

  We finally complete the proof by choosing $\varepsilon_k$.
  From \eqref{eq:enclosing7} we deduce that for any integer $k>0$ there is a small number $\varepsilon_k>0$, so that not only $\varepsilon_k<\bar{\varepsilon}$ (for $S$ being a surface) but also
  \begin{equation}\label{eq:enclosing5}
    \varepsilon_k < 1/k,
  \end{equation}
  such that for all $i=1,\dots,N$,
  $$\left|\int_{S_{\varepsilon_k,k,i}}|H| - \frac{1}{2}\ell(\partial M_i)\int_{\gamma_k}|\kappa|ds \right| \leq \frac{1}{k}.$$
  Using this estimate with $\sum_i\ell(\partial M_i)=\ell(M)$, and recalling \eqref{eq:teardrop3}, we obtain
  \begin{equation}\label{eq:enclosing8}
    \lim_{k\to\infty}\sum_{i=1}^N\int_{S_{\varepsilon_k,k,i}}|H| = \frac{1}{2}\ell(\partial M) \lim_{k\to\infty}\int_{\gamma_k}|\kappa|ds = \frac{\pi}{2}\ell(\partial M).
  \end{equation}
  We thus conclude that $\widetilde{\Sigma}_k:=\Sigma_{\varepsilon_k,k}$ is the desired surface, since \eqref{eq:enclosing4} and \eqref{eq:enclosing5} imply \eqref{eq:enclosing1}, while \eqref{eq:enclosing3} and \eqref{eq:enclosing8} imply \eqref{eq:enclosing2}.
\end{proof}

\begin{remark}
  Two-dimensionality of surfaces is essential in our argument.
  Indeed, if one attempts to extend our argument to higher dimensions, then an issue occurs in the singular limit near the boundary.
  For simplicity we consider an embedded submanifold of codimension one, $M^m\subset\mathbf{R}^{m+1}$, whose boundary has a locally flat part.
  Then the closed surface $\Sigma_\varepsilon$ defined by the $\varepsilon$-neighborhood locally looks like half of the cylinder $\Sigma_\varepsilon=\varepsilon\mathbf{S}^1\times[0,1]^{m-1}$, where we have $|H|\sim \varepsilon^{-1}$ and volume $\sim \varepsilon$ and hence $\int_{\Sigma_\varepsilon}|H|^{m-1}=O(\varepsilon^{2-m})\to\infty$ if $m>2$.
  (Recall that the exponent $m-1$ of the integrand naturally arises in view of scaling, cf.\ \eqref{eq:MenneScharrer}.)
  Our method is thus not directly applicable to higher dimensions.
\end{remark}

\section{Nonexistence theorem in the Plateau-Douglas problem}\label{sec:Plateau-Douglas}

In this section we discuss the Plateau-Douglas problem, also known as the Douglas problem or simply the (general) Plateau problem, focusing on the simplest case of codimension one.
The Plateau-Douglas problem asks whether a connected minimal surface $M\hookrightarrow\mathbf{R}^3$ spans a given boundary contour $\Gamma\subset\mathbf{R}^3$ that consists of mutually disjoint smooth Jordan curves.

In this problem the existence of a connected solution sensitively depends on the geometry of $\Gamma$, in contrast to the case that $\Gamma$ is connected (the classical Plateau problem).
A simple phenomenological example is a catenoid-shaped soap film spanning two parallel circular wires; such a connected film exists if the wires are close to each other, but the film will pinch off if the wires are pulled apart.
Up to rescaling, it is equivalent to shrink the circles while keeping the distance, as in Figure \ref{fig:doubly_connected}.
(We remark that Figure \ref{fig:doubly_connected} describes only stable catenoids; however, regardless of the stability, the boundary admits no catenoid below a critical radius.)

\begin{figure}[htbp]
  \includegraphics[width=80mm]{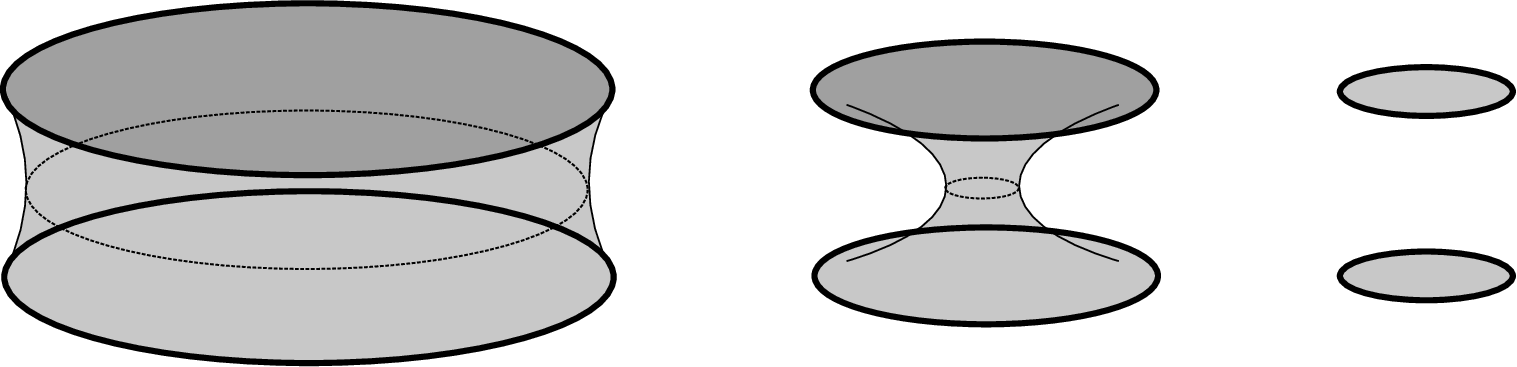}
  \caption{Pinching of a catenoid that spans shrinking circles.}
  \label{fig:doubly_connected}
\end{figure}

Up to now some sufficient conditions on $\Gamma$ for existence of connected minimal surfaces are known, including the celebrated Douglas condition (see e.g.\ \cite[Section 8]{Dierkes2010}), but they do not completely characterize the existence.
For this reason there are also many studies to explore nontrivial necessary conditions.

Our result \eqref{eq:constants} combined with \eqref{eq:ToppingconstantLBimproved} and the fact that $d(\partial M)\leq d(M)$ directly implies a nonexistence criterion, which also explains certain pinching phenomena:

\begin{corollary}\label{cor:nonexistence}
  Let $\Gamma\subset\mathbf{R}^3$ be a boundary contour such that
  \begin{equation}\label{eq:newcondition}
    d(\Gamma)>8\ell(\Gamma).
  \end{equation}
  Then there exists no connected compact minimal surface spanning $\Gamma$.
\end{corollary}

\begin{remark}
  Recall that $8$ is expected to be replaced by $1/2$ in view of the Topping conjecture.
  Up to a universal constant, Menne-Scharrer's estimate \eqref{eq:MenneScharrer} also gives the same type of criterion since $d\leq d_\mathrm{int}$.
\end{remark}

In the rest of this section we observe that this kind of criterion is not covered by existing criteria, by exhibiting some concrete examples.

After Nitsche's pioneering studies, many authors establish several kinds of nonexistence criteria (see e.g.\ \cite{Nitsche1964a,Nitsche1964,Nitsche1965,Bailyn1967,Nitsche1968,Hildebrandt1972,Osserman1975,Li1984,Dierkes1990,Rossman1998,Lopez2001,Dierkes2005,Lopez2017} and also \cite[Chapter 4]{Dierkes2010a}), most of which are based on the maximum principle (except for some results using special geometry).
Roughly speaking, such results assert that no connected solution spans a boundary being ``divided into two parts far from each other''.
For example, a well-known criterion of cone type, which is first obtained by Hildebrandt \cite{Hildebrandt1972} and later improved by Osserman-Schiffer \cite{Osserman1975}, ensures that no solution exists if $\Gamma$ is the union of $\Gamma_1$ and $\Gamma_2$ (not necessarily connected) such that
\begin{equation}\label{eq:unidirectional}
  \mbox{$\Gamma_1\subset K\cap\{z>0\}$ and $\Gamma_2\subset K\cap\{z<0\}$ for $K:=\{x^2+y^2<z^2\sinh^2\tau\}$},
\end{equation}
where $\tau$ is the unique positive solution to $\cosh\tau=\tau\sinh\tau$.
The prefactor $\sinh^2\tau\approx2.27$ is known to be optimal, cf.\ \cite{Osserman1975}.
This kind of criterion is sufficient for explaining the pinching of a catenoid.
However, it does not necessarily cover the case where many circles are simultaneously pulled apart in some random directions, or equivalently, shrunk into several randomly distributed points, cf.\ Figure \ref{fig:uniform}.

\begin{figure}[htbp]
  \includegraphics[width=60mm]{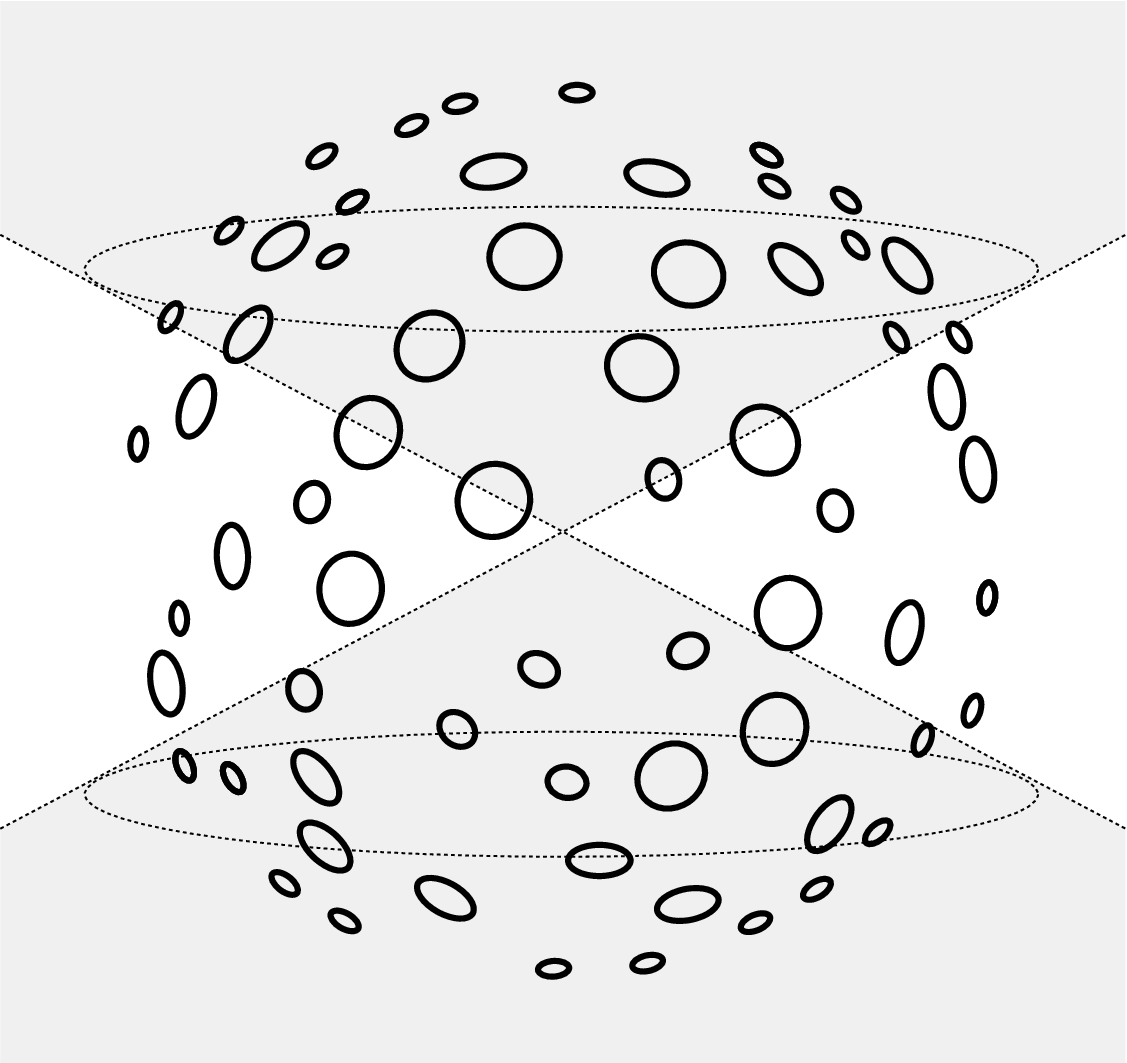}
  \caption{A boundary contour $\Gamma$ consisting of randomly distributed small circles, which cannot be separated by the cone $K$.}
  \label{fig:uniform}
\end{figure}

Recently, White gave an essentially different criterion \cite[Corollary 9]{White2016} (see also \cite{Seo2017}) by making use of an extended monotonicity theorem for density (due to \mbox{Gromov} \cite[Theorem 8.2.A]{Gromov1983} and rediscovered in \cite{Ekholm2002}).
It asserts that no solution spans $\Gamma$ if there is a decomposition $\Gamma=\Gamma_1\cup\Gamma_2$ such that
\begin{equation}\label{eq:White}
  \mathrm{dist}(\Gamma_1,\Gamma_2) > \frac{1}{\pi}\ell(\Gamma).
\end{equation}
This looks similar to our one \eqref{eq:newcondition} as both the conditions require smallness of length, and in fact shares the common feature of being ``less unidirectional''.

We shall discuss ``multi-directionality'' through a concrete example of a boundary contour that is covered by both \eqref{eq:newcondition} and \eqref{eq:White} but not by \eqref{eq:unidirectional}.

\begin{example}[Small circles on the sphere: Fixed centers]\label{example1}
  Let $X$ be any finite set in the unit sphere $\mathbf{S}^2\subset\mathbf{R}^3$ containing more than one point.
  Let $r_X$ be the packing radius of $X$ in $\mathbf{S}^2$, i.e., half of the minimal geodesic distance of two distinct points in $X\subset\mathbf{S}^2$, and also $R_X$ the covering radius, i.e., the minimal $R>0$ such that the (closed) $R$-neighborhood of $X$ in $\mathbf{S}^2$ covers the whole sphere $\mathbf{S}^2$.
  Let $\Gamma_\varepsilon$ be the union of all the mutually disjoint geodesic circles in $\mathbf{S}^2$ of radius $\varepsilon<r_X$ centered at the points of $X$.
  Then our condition \eqref{eq:newcondition} is satisfied by $\Gamma_\varepsilon$ for any small $\varepsilon$, since $d(\Gamma_\varepsilon)\to d(X)>0$ while $\ell(\Gamma_\varepsilon)\to0$ as $\varepsilon\to0$.
  Similarly, White's condition \eqref{eq:White} is also satisfied for any small $\varepsilon$, since if we take $\Gamma_{\varepsilon,1}$ to be one circle and $\Gamma_{\varepsilon,2}$ all the remaining circles, then $\lim_{\varepsilon\to0}\mathrm{dist}(\Gamma_{\varepsilon,1},\Gamma_{\varepsilon,2})\geq\frac{2r_X}{\pi}$.
  However, the cone condition \eqref{eq:unidirectional} may be violated depending on the choice of $X$.
  More precisely, if $X$ is sufficiently dense, say $R_X\leq\frac{1}{100}$, then $\Gamma_\varepsilon$ cannot be separated by any cone $K'$ congruent to $K$; indeed, if $K'$ separates $\Gamma_\varepsilon\subset\mathbf{S}^2$, then the center of $K'$ must lie inside $\mathbf{S}^2$, and this constraint implies a universal lower bound of the (geodesic) inradius of $\mathbf{S}^2\setminus K'$ so that in particular $R_X>\frac{1}{100}$ must hold.
\end{example}

In addition, our condition \eqref{eq:newcondition} is qualitatively independent from \eqref{eq:White}.
To observe this we shall perturb a given $\Gamma$ by adding very small circles.
For such a perturbation, condition \eqref{eq:newcondition} is generally rigid, but condition \eqref{eq:White} may be violated depending on the positions since its left-hand side may considerably decrease.

From another point of view we may say that \eqref{eq:newcondition} is in a sense ``more multi-directional'' than \eqref{eq:White}, as is seen in the next example, which is similar to Example \ref{example1} but allows more densely distributed circles.

\begin{example}[Small circles on the sphere: Increasing centers]\label{example2}
  Let $X_\varepsilon$ be an $\varepsilon$-net of $\mathbf{S}^2$ with $\varepsilon\ll1$; namely, if we let $r_{X_\varepsilon}$ (resp.\ $R_{X_\varepsilon}$) denote the packing (resp.\ covering) radius as in Example \ref{example1}, then $\varepsilon/2\leq r_{X_\varepsilon}<R_{X_\varepsilon}\leq \varepsilon$.
  Let $\Gamma_\varepsilon$ be the union of all the geodesic circles centered at the points in $X_\varepsilon$ of radius $\varepsilon^{2+\alpha}$ ($\ll\varepsilon$) for $\alpha\in(0,1)$.
  Note that the cardinality of $X_\varepsilon$ is of the form $|X_\varepsilon|\sim\varepsilon^{-2}$, and hence $\ell(\Gamma_\varepsilon)\sim\varepsilon^{2+\alpha}\cdot\varepsilon^{-2}=\varepsilon^{\alpha}$.
  Then \eqref{eq:newcondition} is still satisfied since $d(\Gamma_\varepsilon)\to2$ while $\ell(\Gamma_\varepsilon)\to0$ as $\varepsilon\to0$.
  However, \eqref{eq:White} is violated since $\mathrm{dist}(\Gamma_{\varepsilon,1},\Gamma_{\varepsilon,2})\lesssim\varepsilon$ holds for an arbitrary choice of decomposition $\Gamma_{\varepsilon,1}$ and $\Gamma_{\varepsilon,2}$, so that $\mathrm{dist}(\Gamma_{\varepsilon,1},\Gamma_{\varepsilon,2})/\ell(\Gamma_\varepsilon)\lesssim\varepsilon^{1-\alpha}\to0$.
\end{example}

We finally remark that the above multi-directional criteria \eqref{eq:White} and \eqref{eq:newcondition} require smallness of $\ell(\Gamma)$, although unidirectional results such as \eqref{eq:unidirectional} need not require any smallness of first order.
To find a (qualitative) common roof, it would be a natural direction to explore a ``zeroth order'' multi-directional criterion.
For example:

{\em Given an arbitrary finite set $X\subset\mathbf{R}^3$ containing more than one point, can one find $\varepsilon=\varepsilon(X)>0$ such that if the boundary $\partial M$ of a connected compact minimal surface $M$ is contained in the $\varepsilon$-neighborhood $U_\varepsilon(X)$ of $X$, then there is a connected component of $U_\varepsilon(X)$ that does not intersect with $\partial M$?}

Notice that such an $\varepsilon$ clearly exists if $X$ is separated by the (open) cone $K$ in \eqref{eq:unidirectional} up to a rigid motion, since so is $U_\varepsilon(X)$ for any small $\varepsilon$ in that case.
Notice also that such an inclusion property always holds if $X$ has a sufficiently small number of points.
However, it seems still open for $X$ being more uniformly distributed.
A next step would be to ask which quantity of $X$ the choice of $\varepsilon$ depends on.

\appendix

\section{Lower bound of Topping's constant}\label{sec:appendix}

In this appendix we observe that how Brendle's result improves Topping's lower bound $C_T(n)\geq\pi/32$ for low codimensions $n\leq5$; namely,
\begin{align}\label{eq:ToppingconstantLBimproved}
  C_T(n) \geq
  \begin{cases}
    \pi/16 & (n=3,4),\\
    \pi/24 & (n=5),\\
    \pi/32 & (n\geq6).
  \end{cases}
\end{align}

To this end we shall first recall the Michael-Simon Sobolev inequality, which is a key ingredient in Topping's argument.
Throughout this section we focus on two-dimensional closed surfaces $M\hookrightarrow\mathbf{R}^n$ for our purpose.
For such a surface, it is known (due to L.\ Simon, cf.\ \cite[Lemma 2.1]{Topping2008}) that
\begin{equation}\label{eq:MichaelSimon}
  \sigma\Big(\int_M|f|^2\Big)^\frac{1}{2} \leq \int_M|\nabla f|+\int_M|2Hf|
\end{equation}
holds for every (smooth $f$ and hence also) $f\in W^{1,1}(M)$, where
\begin{equation}\label{eq:MichaelSimonconstant1}
  \sigma=\sqrt{2\pi}.
\end{equation}
On the other hand, Brendle's recent result \cite[Theorem 1]{Brendle2021} implies (in the special case of $\dim{M}=2$ and $\partial M=\emptyset$) that inequality \eqref{eq:MichaelSimon} holds with the constant
\begin{equation}\label{eq:MichaelSimonconstant2}
  \sigma=\sigma(n)=\sqrt{\frac{8\pi}{n-2}} \qquad \mbox{if}\ n\geq4,
\end{equation}
which is larger than \eqref{eq:MichaelSimonconstant1} up to codimension three ($n\leq5$), and in fact sharp up to codimension two ($n=4$).
Notice that for $n=3$ we can also take $\sigma(4)$ in \eqref{eq:MichaelSimon} thanks to the canonical inclusion $M\hookrightarrow\mathbf{R}^3\hookrightarrow\mathbf{R}^4$.
Here we remark two minor differences between \eqref{eq:MichaelSimon} and \cite[Theorem 1]{Brendle2021}; there appears the prefactor $2$ of $Hf$ only in \eqref{eq:MichaelSimon} but this just comes from our convention to average $H$; also, the right-hand side of \eqref{eq:MichaelSimon} is different from the one $\int_M\sqrt{|\nabla f|^2+|2Hf|^2}$ given in \cite[Theorem 1]{Brendle2021} but it does not matter since $\sqrt{a^2+b^2}\leq|a|+|b|$.

Now we recall Topping's proof of a key lemma \cite[Lemma 1.2]{Topping2008} based on \eqref{eq:MichaelSimon} with \eqref{eq:MichaelSimonconstant1}, in order to clarify how the improved constant \eqref{eq:MichaelSimonconstant2} comes into play.

\begin{lemma}\label{lem:collapsedness}
  Let $M$ be a closed surface in $\mathbf{R}^n$, and let $p\in M$ and $R>0$.
  Let
  \begin{align*}
    m(p,R) := \sup_{r\in(0,R]}\frac{1}{r}\int_{B(p,r)}|H|, \qquad \kappa(p,R) := \inf_{r\in(0,R]}\frac{1}{r^2}V(p,r),
  \end{align*}
  where $B(p,r)$ denotes the intrinsic ball in $M$ of radius $r$ centered at $p$, and $V(p,r)$ denotes the $2$-dimensional volume of $B(p,r)\subset M$.
  Then we have
  \begin{equation*}
    \max\{m(p,R),\kappa(p,R)\}\geq\delta:=\frac{\sigma^2}{16},
  \end{equation*}
  where $\sigma=\sigma(n)\in[\sqrt{2\pi},2\sqrt{\pi}]$ is a constant for which \eqref{eq:MichaelSimon} holds.

  In particular, we can take $\delta=\pi/8$ for all $n\geq3$ due to \eqref{eq:MichaelSimonconstant1}, while $\delta=\pi/4$ (resp.\ $\delta=\pi/6$) is also allowable for $n=3,4$ (resp.\ $n=5$) due to \eqref{eq:MichaelSimonconstant2}.
\end{lemma}

\begin{proof}
  Suppose that $m(p,R)\leq\delta$.
  Then for any $r\in(0,R]$,
  \begin{equation}\label{eq:A1}
    \int_{B(p,r)}|H|\leq \delta r.
  \end{equation}
  Given $\mu>0$, we define a Lipschitz cut-off function $f$ on $M$ by $f\equiv1$ inside $B(p,r)$, $f\equiv0$ outside $B(p,r+\mu)$, and $f:=1-\frac{1}{\mu}(\mathrm{dist}_M(\cdot,p)-r)$ in the annulus $B(p,r+\mu)\setminus B(p,r)$.
  Applying \eqref{eq:MichaelSimon} to this $f$, and dropping $p$ in $V(p,r)$ for simplicity, we have
  \begin{equation*}
    \sigma V(r)^\frac{1}{2} \leq \sigma \|f\|_{L^{2}(M)} \leq \frac{1}{\mu}(V(r+\mu)-V(\mu)) + 2\int_{B(p,r+\mu)}|H|.
  \end{equation*}
  Recall that $V$ is locally Lipschitz, where the Lipschitz constant is bounded in terms of the Ricci curvature.
  Hence, letting $\mu\downarrow0$, we deduce that for a.e.\ $r$,
  \begin{equation}\label{eq:A4}
    \sigma V^\frac{1}{2} \leq \frac{dV}{dr} + 2\int_{B(p,r)}|H|.
  \end{equation}
  Combing this with \eqref{eq:A1}, we have
  \begin{equation}\label{eq:A2}
    V' + 2\delta r - \sigma V^\frac{1}{2} \geq 0.
  \end{equation}
  We now compare this $V$ with the function $v(r):=\delta r^2$, which satisfies that
  \begin{equation}\label{eq:A3}
    v' + 2\delta r -\sigma v^\frac{1}{2} = (4\delta^\frac{1}{2} - \sigma)\delta^\frac{1}{2} r = 0.
  \end{equation}
  Using \eqref{eq:A2}, \eqref{eq:A3}, and the fact that $V(r)/r^2\to\pi$ as $r\downarrow0$, where $\pi$ is the volume of the Euclidean unit $2$-ball, while $v(r)/r^2=\delta=\sigma^2/16\leq\pi/4<\pi$, we conclude that $V(r)>v(r)$ for all $r\in(0,R]$ and hence $\kappa(p,R) = \inf_{r\in(0,R]}V(x,r)/r^2>\delta.$
\end{proof}

Once we get Lemma \ref{lem:collapsedness}, by the completely same covering argument as in \cite[Section 3]{Topping2008} we deduce that
$$d_\mathrm{int}(M) \leq \frac{4}{\delta}\int_M|H|.$$
Then \eqref{eq:ToppingconstantLBimproved} immediately follows thanks to the allowable choices of $\delta$ in Lemma \ref{lem:collapsedness} (and $d\leq d_\mathrm{int}$).

\begin{remark}\label{rem:appendix}
  A core of Topping's argument is obtaining a uniformly positive lower bound for $U:=V/r^2$ by using the structure of differential inequality \eqref{eq:A2}.
  As long as we follow this strategy, the best possible choice of $\delta$ would be at most $\pi/2$ (yielding $C_T\geq\pi/8$).
  To roughly observe this, we first recall that $\sigma=2\sqrt{\pi}$ is sharp in \eqref{eq:MichaelSimonconstant1}, as is seen from \eqref{eq:A4} in the locally flat case.
  On the other hand, it seems not yet clear whether the term $\int_M|2Hf|$ is also sharp or not.
  One might expect that $\int_M|2Hf|$ could be halved (as is observed in \cite{Dalphin2016}) but any smaller one is not allowed due to the case of $M$ being a round sphere and $f\equiv1$.
  Here we temporally suppose that this expectation would be also true; then, inequality \eqref{eq:A2} would turn to the stronger form $V' + \delta r - 2\sqrt{\pi} V^\frac{1}{2} \geq 0$, or equivalently $rU'+2(\sqrt{U}-\sqrt{\pi/4})^2\geq\pi/2-\delta$.
  However, for $\delta>\pi/2$, this constraint does not rule out the possibility that $U$ approaches to zero.
\end{remark}

\bibliography{bibliography}
\bibliographystyle{amsplain}

\end{document}